\documentclass[11pt,reqno]{amsart}
\usepackage{amsfonts}
\usepackage{amsmath}
\usepackage{amssymb}
\usepackage{hyperref}

\usepackage{amsthm}
\usepackage{verbatim}
\theoremstyle{plain}
\newtheorem{theorem}{Theorem}[section]
\newtheorem{proposition}{Proposition}[section]
\newtheorem{lemma}{Lemma}[section]
\newtheorem{definition}{Definition}[section]
\newtheorem{remark}{Remark}[section]
\newtheorem{conjecture}{Conjecture}
\newtheorem{corollary}{Corollary}[section]
\newtheorem{example}{Example}[section]
\newtheorem*{conja}{Conjecture A}
\newtheorem*{conjb}{Conjecture B}
\usepackage[noabbrev,capitalize]{cleveref}

\usepackage[]{cite}
\usepackage{tikz}

\title{Haglund's positivity conjecture for multiplicity one pairs}
\author{Aritra Bhattacharya}
\address{The Institute of Mathematical Sciences, A CI of Homi Bhabha National Institute, Chennai 600113, India}
\email{baritra@imsc.res.in.}
\date{\today}
\begin{document}
	
	\maketitle
	
	\begin{abstract}
		Haglund's conjecture states that $\dfrac{\langle J_{\lambda}(q,q^k),s_\mu \rangle}{(1-q)^{|\lambda|}} \in \mathbb{Z}_{\geq 0}[q]$ for all partitions $\lambda,\mu$ and all non-negative integers $k$, where  $J_{\lambda}$ is the integral form Macdonald symmetric function and $s_\mu$ is the Schur function. This paper proves Haglund's conjecture in the cases when the pair $(\lambda,\mu)$ satisfies $K_{\lambda,\mu}=1$ or $K_{\mu',\lambda'}=1$ where $K$ denotes the Kostka number. We also obtain some general results about the transition matrix between Macdonald symmetric functions and Schur functions.	
	\end{abstract}
	
	\section{Introduction}
	
	The Macdonald symmetric functions $P_{\lambda}(q,t): \lambda \in Par$ are a remarkable family of symmetric functions depending on two parameters $q$ and $t$, indexed by the set of partitions $Par$. They simultaneously generalize many known bases of symmetric functions.  
	
	We denote by $K^{(1)}_{\lambda,\mu}(q,t)$ the coefficient of the monomial symmetric function $m_{\mu}$ in the monomial expansion of $P_{\lambda}(q,t)$. This is a rational function in $q$ and $t$ which has the Kostka number $K_{\lambda,\mu}$ as its limit. There are various expressions for $K^{(1)}_{\lambda,\mu}(q,t)$ in the literature: in \cite{Mac95} a tableaux formula is given, in \cite{HHL04}, \cite{HHL06} a formula in terms of nonattacking fillings is found and in \cite{RY08} a formula in terms of alcove walks is given. A nice survey of some of the monomial expansions can be found in \cite{GR21} and \cite{GR21sup}. 
	
	In contrast, very little is known about the Schur expansion of $P_{\lambda}(q,t)$. Some particular Schur coefficients of the integral form Macdonald polynomials $J_{\lambda}(q,t)$ were found in \cite{Yoo12} and \cite{Yoo15}. The integral form Macdonald polynomial $J_{\lambda}(q,t)$ is a certain normalization of the $P_{\lambda}(q,t)$, whose monomial coefficients are in $\mathbb{Z}[q,t]$. 
	
	In \cite{HHL04} it was shown that the coefficient $\langle J_{\lambda}(q,t), h_{\mu} \rangle$ of $m_\mu$ in $J_{\lambda}(q,t)$ has the following positivity property: \begin{eqnarray}
		\dfrac{\langle J_{\lambda}(q,q^k),h_{\mu} \rangle}{(1-q)^{|\lambda|}} \in \mathbb{Z}_{\geq 0}[q] \quad \forall \ k \in \mathbb{Z}_{\geq 0}  \quad \forall \ \lambda,\mu \in Par
	\end{eqnarray}  J. Haglund  \cite{Hag10} conjectured that the above equation holds true if $h_{\mu}$ is replaced with $s_{\mu}$: \begin{conja}[Haglund] 
		For partitions $\lambda$ and $\mu$,
		\begin{equation*}
			\dfrac{\langle J_{\lambda}(q,q^k),s_\mu \rangle}{(1-q)^{|\lambda|}} \in \mathbb{Z}_{\geq 0}[q] \quad \forall \ k \in \mathbb{Z}_{\geq 0} \tag{Hag($\lambda,\mu$)} \label{Hag}
		\end{equation*}

	\end{conja}

	M. Yoo showed this is true in some special cases in \cite{Yoo12}, \cite{Yoo15} by obtaining explicit formulas for the coefficients. Some particular ones among these cases are when $\lambda$ has only one row, when $\lambda$ is of hook shape with $\ell(\lambda) \geq \lambda_1-2$, when $\lambda$ has at most $2$ columns, when $\mu$ is of hook shape, or when $\lambda,\mu$ both have length at most 2. 
	
	This paper extends Yoo's results in a different direction-- showing Haglund's conjecture to be true for all `multiplicity one pairs' i.e, for $(\lambda,\mu)$ such that $K_{\lambda,\mu}=1$ or $K_{\mu',\lambda'}=1$.

	The following is the main result of this paper. 
	\begin{theorem}
		\emph{\ref{Hag}} holds true for all pairs $\lambda,\mu$ such that either $K_{\lambda,\mu}=1$ or $K_{\mu',\lambda'}=1$
	\end{theorem}
	
	In fact, what we establish is a dual version of Haglund's conjecture which can be easily seen to be equivalent to the original. Namely, let $k_{\lambda,\mu}(q,t) = \langle s_{\lambda}, J_{\mu}(q,t) \rangle_{q,t}$ where $\langle \ , \ \rangle_{q,t}$ is the $(q,t)$-deformation of the Hall scalar product \cite{Mac95}. Then
	\begin{conjb}[dual Haglund's conjecture] 
		For partitions $\lambda,\mu$ 
		\begin{equation*}
			\dfrac{k_{\lambda,\mu}(t^k,t)}{(1-t)^{|\lambda|}} \in \mathbb{Z}_{\geq 0}[t] \quad \forall \ k \in \mathbb{Z}_{\geq 0} \tag{Hag$^\prime(\lambda,\mu$)} 
		\end{equation*}
	\end{conjb}

	We also define another $q,t$-analogue $K^{(2)}_{\lambda,\mu}(q,t)$ of the Kostka numbers by: \begin{eqnarray}
		s_{\lambda} = \sum_{\mu \in Par}K^{(2)}_{\lambda,\mu}(q,t)P_{\mu}(q,t) 
		& \forall \ \lambda \in Par 
	\end{eqnarray}
	$K^{(2)}_{\lambda,\mu}(q,t)$ coincides with $k_{\lambda,\mu}(q,t)$ upto a constant. Analogous to $K^{(1)}_{\lambda,\mu}(q,t)$ this is a rational function in $q,t$ and $K^{(2)}_{\lambda,\mu}(0,t)=K_{\lambda,\mu}(t)$, the Kostka-Foulkes polynomial. The matrix $(K^{(2)}_{\lambda,\mu})_{\lambda,\mu \, \vdash \, n}$ has the interesting property that \cite{Mac95} (Chapter VI, ($5.1^{\prime\prime}$)) \begin{eqnarray}
		K^{(2)}_{\lambda,\mu}(q,t) = (K^{(2)})^{-1}_{\mu',\lambda'}(t,q)
	\end{eqnarray}
	Since the entries of $(K^{(2)})^{-1}$ are the Schur coefficients of the $P_{\lambda}(q,t)$, finding $K^{(2)}$ is roughly equivalent to finding these.    
	
	We use Macdonald's description of $K^{(1)}_{\lambda,\mu}(q,t)$ and basic properties of $P_{\lambda}(q,t)$ to find some general reduction principles. In particular, we show that (\autoref{rowmult}) if for a pair of partitions $(\lambda,\mu)$, the $r$-th partial sum $\lambda_1+\ldots+ \lambda_r = \mu_1+\ldots + \mu_r$ then for $i=1,2$, \[
	(K^{(i)})^{\pm 1}_{\lambda,\ \mu} = (K^{(i)})^{\pm 1}_{\lambda^1,\ \mu^1}(K^{(i)})^{\pm 1}_{\lambda^2,\ \mu^2}
	\] where $\lambda^1 = (\lambda_1,\ldots,\lambda_r), \mu^1 = (\mu_1,\ldots,\mu_r), \lambda^2 = (\lambda_{r+1},\ldots), \mu^2 = (\mu_{r+1},\ldots)$. 
	A dual version of this is given in \autoref{colmult}, where instead of breaking the Young diagrams across a row, we break across a column. 
	
	Let $\lambda,\mu \in Par$ such that $\lambda \geq \mu$. We can always decompose $\lambda = (\lambda^1,\ldots,\lambda^r)$ and $\mu = (\mu^1,\ldots,\mu^r)$ where $\lambda^1,\ldots \lambda^r,\mu^1,\ldots,\mu^r$ are partitions such that for $i\in \{1,\ldots r\}$, $\lambda^i \geq \mu^i$ and for $i \in \{1,\ldots r-1\}$, $l(\lambda^i)=l(\mu^i)$. Then for $i=1,2$, \begin{eqnarray}
		(K^{(i)})^{\pm 1}_{\lambda,\mu}(q,t) = \prod_{j=1}^{r}(K^{(i)})^{\pm 1}_{\lambda^j,\mu^j}(q,t)
	\end{eqnarray}
	
	For $j\in \{1,\ldots, r-1\}$ each pair $(\lambda^j,\mu^j)$ contains a common rectangle of row length atleast $\lambda^{j+1}_1$. Proving the dual Haglund conjecture is equivalent to proving it for each pair $(\lambda^j,\mu^j)$ with these rectangles removed (\autoref{cor:rectangleinsert}). We can repeat this process to each pair till no more decomposition is possible.

	This allows us to only look at pairs of partitions $(\lambda,\mu)$ which are \textsl{irreducible} i.e, satisfying $\lambda \geq \mu$ and  $\lambda_1+\ldots+\lambda_i > \mu_1+\ldots+\mu_i$ for $1\leq i \leq \ell(\lambda)$. Berenshtein and Zelevinskii in \cite{BZ90} gave a criterion for a pair of partitions $(\lambda,\mu)$ to satisfy $K_{\lambda,\mu}=1$. This can be restated as \begin{theorem}\label{BZtheorem}
		Let $\lambda,\mu$ be an irreducible pair. Then $K_{\lambda,\mu} = 1$ if and only if either \begin{enumerate}
			\item $\lambda = (m^n)$ for some $m \in \mathbb{Z}_{\geq 1}$, $n \in \mathbb{Z}_{\geq 0 }$ and $l(\mu)=n+1$
			
			or
			\item $\lambda = (n)$ and $\mu \vdash n$, $\mu \neq \lambda$ for some $n\in \mathbb{Z}_{\geq 1}$  
		\end{enumerate} 
		
	\end{theorem}
	
	We find (\autoref{complement}) that if $\lambda,\mu$ are partitions contained in a rectangle $(m^{n+1})$ then 
	for $i=1,2$
	\begin{eqnarray}
		(K^{(i)})^{\pm 1}_{\lambda,\mu} = (K^{(i)})^{\pm 1}_{\lambda^c,\mu^c} 
	\end{eqnarray}    
	where $\lambda^c,\mu^c$ denote the complements of the partitions $\lambda,\mu$ inside the rectangle $(m^{n+1})$. Since the complement of the partition $(m^n)$ is $(m)$, this reduces all computation of $K^{(2)}_{\lambda,\mu}$ when $K_{\lambda,\mu}=1$ to the case where $\lambda$ is a single row. 
	
	In this case we calculate $K^{(2)}_{\lambda,\mu}$ by use of the Cauchy identity and Macdonald's principal specialization formula (\autoref{k_n_mu})
	
	Finally, in \autoref{sec:whenKlammu=1} we finish proving the dual version of Haglund's conjecture when $K_{\lambda,\mu}=1$ by analyzing the normalization factors for the integral form Macdonald polynomials. By exploiting duality of Macdonald polynomials (\autoref{Kdual}) we get the proof when $K_{\mu',\lambda'}=1$.
	
	Our reduction principles may also be applied to instances where the final reduction is to pairs $(\lambda,\mu)$ of the form studied by Yoo. Therefore explicit formulas for $k_{\lambda,\mu}(q,t)$ can be obtained in a wider set of cases.

	
	An interesting feauture of our formulas and many of Yoo's formulas is that $K^{(2)}_{\lambda,\mu}(q,t)$ can be written as a sum of $K_{\lambda,\mu}$ many terms, where each term is a product with factors of the form $(t^a-q^bt^c)$ in the numerator and $(1-q^dt^e)$ in the denominator with $a,b,c,d,e \in \mathbb{Z}_{\geq 0}$. This is especially striking in the $K_{\mu',\lambda'}=1$ case where $K_{\lambda,\mu}$ could be $>1$.  

	\subsection*{Acknowledgments} We would like to thank R. Venkatesh for organizing the Workshop on Macdonald Polynomials 2021. We are very grateful to A. Ram for his inspiring lectures and discussions. Finally, a lot of thanks to S. Viswanath for numerous discussions, going through the whole paper and giving helpful suggestions.          
	
	\section{Definitions}
	
	We first review some basic definitions in the theory of symmetric functions. The main reference for this section is \cite{Mac95}.
	
	\subsection{Partitions}
	By a partition we mean a sequence $\lambda = (\lambda_1,\lambda_2,...)$ of non-negative integers which is weakly decreasing and has a finite sum. By $Par$ we denote the set of all partitions. A partition $\lambda$ may be realized as a set of boxes upper left justified arranged in rows so that there are $\lambda_i$ boxes in row $i$, this is called the Young diagram of $\lambda$. For a partition $\lambda$ the conjugate partition $\lambda'$ is the partition obtained by interchanging the rows and columns in the diagram of $\lambda$. The length of a partition $\lambda$ is $\ell(\lambda) = \lambda'_1$ and the size is $|\lambda|=\sum \lambda_i$
	
	The set of partitions is partially ordered by the dominance order defined by \begin{eqnarray}
		\lambda \geq \mu \iff  \lambda_1+...+\lambda_i \geq \mu_1+...+\mu_i \ \forall \ i\geq 1, \text{and} \ |\lambda|=|\mu|
	\end{eqnarray}
	
	For a box $x=(r,c)$ in the diagram of $\lambda$, where $r,c$ are the row and column numbers of $x$ with row numbering starting from 1 in top row and column numbering starting from 1 in the leftmost column, we write $x\in \lambda$ and define the arm length $a_{\lambda}(x)$, coarm length $a'_\lambda(x)$, leg length $l_{\lambda}(x)$ and the coleg length $l'_{\lambda}(x)$ of $x$ by: \begin{eqnarray}
		a_{\lambda}(x) = \lambda_r - c & a'_{\lambda}(x) = c-1 \\ l_{\lambda}(x) = \lambda'_c - r & l'_{\lambda}(x) = r-1
	\end{eqnarray}
	
	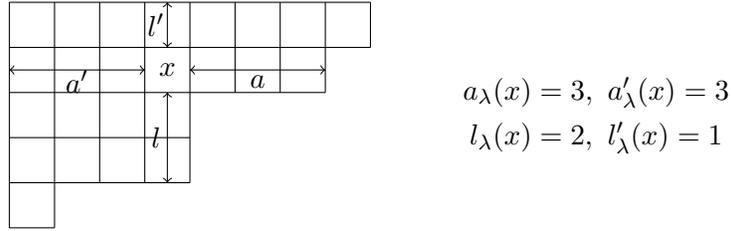
\begin{figure}[h]
		\centering
		\begin{tikzpicture}[yscale=-0.6, xscale=0.6]
			\draw (0,0) -- (8,0);
			\draw (0,1) -- (8,1);
			\draw (0,2) -- (7,2);
			\draw (0,3) -- (4,3);
			\draw (0,4) -- (4,4);
			\draw (0,5) -- (1,5);
			
			\draw (0,0) -- (0,5);
			\draw (1,0) -- (1,5);
			\draw (2,0) -- (2,4);
			\draw (3,0) -- (3,4);
			\draw (4,0) -- (4,4);
			\draw (5,0) -- (5,2);
			\draw (6,0) -- (6,2);
			\draw (7,0) -- (7,2);
			\draw (8,0) -- (8,1);			
			\draw[<->] (4,1.5) -- (7,1.5);
			\draw (5.5,1.75) node{$a$} ;
			
			\draw[<->] (3,1.5) -- (0,1.5);
			\draw (1.5,1.75) node{$a'$} ;
			
			\draw[<->] (3.5,1) -- (3.5,0);
			\draw (3.25,0.5) node{$l'$} ;
			
			\draw[<->] (3.5,2) -- (3.5,4);
			\draw (3.25,3) node{$l$} ;

			\draw (3.5,1.5) node{$x$};
			
			\draw (13,2) node{$a_\lambda(x)=3, \ a'_\lambda(x)=3$}; 
			\draw (13,3) node{$l_\lambda(x)=2, \ l'_\lambda(x)=1$};
		\end{tikzpicture}
		\caption{Young diagram for $\lambda = (8, 7, 4, 4,  1)$ and arm, coarm, leg, coleg lengths of $x = (2,4)$}
		\label{fig:Youngdiagram}
	\end{figure}

	When the partition is clear, we drop $\lambda$ from the notation and simply write $a(x),a'(x),l(x),l'(x)$.
	
	The content of a box $x\in \lambda$ is $c(x) = a'(x)-l'(x)$. The hook length of the box $x$ is $h(x)=a(x)+1+l(x)$. For a partition $\lambda$, define $n(\lambda) = \sum_{x\in \lambda}l'(x)$

	\subsection{Symmetric Functions}
	Let $q$ and $t$ be complex variables. We work in the ring $\Lambda$ of symmetric functions in infinitely many variables $x_1,x_2,...$ with coefficients from $\mathbb{C}(q,t)$ 
	
	The space of symmetric functions has a few known bases indexed by $\lambda \in Par$, such as the monomial basis $(m_{\lambda})$, the elementary basis $(e_{\lambda})$, the homogeneous basis $(h_{\lambda})$, the power sum basis $(p_{\lambda})$, and the Schur basis $(s_{\lambda})$. Their definitions can be found in \cite{Mac95}. 
	
	For a partition $\lambda$ and for $i\geq 1$ let $m_{i}(\lambda)$ be the number of parts of $\lambda$ equal to $i$ and  $z_{\lambda}= \prod_{i=1}^{\ell(\lambda)}i^{m_i(\lambda)}m_{i}(\lambda)!$. The Hall scalar product on $\Lambda$ denoted by $\langle \ , \ \rangle$ is defined by $\langle p_{\lambda} , p_{\mu} \rangle = z_{\lambda} \delta_{\lambda,\mu}$. We also have $\langle s_{\lambda} , s_{\mu} \rangle = \delta_{\lambda,\mu}$, $\langle h_{\lambda} , m_{\mu} \rangle = \delta_{\lambda,\mu}$. 
	
	\subsection{Plethysm}
	
	The power sums $p_k(X)=\sum_{i \geq 0 } x_i^k$, $k \geq 1$ form an algebraically independent generating set for $\Lambda$ over $\mathbb{C}(q,t)$. For a formal series of rational functions $E=E(z_1,z_2,\ldots)$ we denote by $p_k[E]$ the plethystic substitution of $E$ into $p_k$, defined to be $E(z_1^k,z_2^k,\ldots)$. In other words, $p_k[E]$ is the result of replacing each variable in $E$ by it's $k$-th power, and the constants are kept as it is. Since any symmetric function $f\in \Lambda$ is a polynomial in the $p_k$, we can define $f[E]$ as the unique $\mathbb{C}$-algebra homomorphism extending $p_k \mapsto p_k[E]$. Note that, we are not extending as $\mathbb{C}(q,t)$-algebra homomorphism, i.e, $q,t$ are also treated as variables. When using plethysm, we write alphabets $X = \{x_1,x_2,...\}$, $Y=\{y_1,y_2,...\}$ as $X= \sum x_i$, $Y = \sum y_i$. So $p_k[X]=\sum x_i^k = p_k(X)$ for all $k$, and by extension, $f[X] = f(X)$ for any symmetric function $f$.  
	
	In particular, \begin{eqnarray}
		p_{k}\bigg[ X \dfrac{1-q}{1-t}\bigg] = p_k(X)\dfrac{1-q^k}{1-t^k} \\ p_k\bigg[\dfrac{1-q}{1-t}\bigg] = \dfrac{1-q^k}{1-t^k}
	\end{eqnarray}
	
	For more on plethysm see \cite{Hai99}, \cite{Hag06}.
	
	\subsection{Tableaux}
	A semistandard Young tableau or a tableau $T$ of shape $\lambda$ is a filling of the diagram $\lambda$ with entries from $1,2,...$ such that the entries are weakly increasing left to right along the rows and strictly increasing top to bottom along columns. If for each $i\geq 1$ there are $\mu_i$ many $i$'s in $T$ then $T$ is said to have content $\mu=(\mu_1,\mu_2,...)$.  We denote the set of tableaux of shape $\lambda$, content $\mu$ by $SSYT(\lambda,\mu)$. The cardinality of this set is the Kostka number $K_{\lambda,\mu}$. 
	
	We have \begin{eqnarray}
		s_{\lambda} = \sum_{\mu} K_{\lambda,\mu} m_{\mu}
	\end{eqnarray}
	
	By a matrix indexed by partitions we always mean that the rows and columns are indexed by partitions, and the indexing is compatible with dominance ordering i.e, the partitions are listed in sequence such that if $\lambda \geq \mu$ in dominance order then $\lambda$ precedes $\mu$ in the sequence and if $|\lambda|<|\mu|$ then $\lambda$ precedes $\mu$. These are infinite matrices. For a matrix $M$ indexed by partitions, we call it upper-unitriangular if $M_{\lambda,\mu}=0$ unless $\mu \leq \lambda$ and $M_{\lambda,\lambda}=1$. The set of upper unitriangular matrices form a group.   
	
	We will denote by $K$ the matrix indexed by partitions whose $(\lambda,\mu)$th entry is $K_{\lambda,\mu}$. Then $K$ is upper-unitriangular.

	\subsection{Macdonald Symmetric Functions}
	
	The scalar product $\langle \ , \ \rangle_{q,t}$ on $\Lambda$ is defined by setting  \begin{eqnarray}
		\langle f \,, g \rangle_{q,t} = \bigg\langle f[X] ,  g\bigg[X\dfrac{1-q}{1-t}\bigg] \bigg\rangle  
	\end{eqnarray} where $\langle \ , \ \rangle$ on the right hand side is the Hall scalar product.
	
	The Macdonald $P$-functions $(P_\lambda[X;q,t])_{\lambda \in Par}$ are characterized by the following two properties: \begin{eqnarray}
		P_{\lambda}[X;q,t] - m_{\lambda}[X] \in \sum_{\mu < \lambda}\mathbb{C}(q,t)m_\mu[X]  \\ \langle P_{\lambda}[X;q,t], P_\mu[X;q,t] \rangle_{q,t} = 0  \text{ if } \lambda \neq \mu
	\end{eqnarray} 
	where $\leq$ is the dominance order on partitions. 
	
	When the variables are clear, we may drop the $X$ from the notation and just write $P_{\lambda}(q,t)$. 
	
	Let $(Q_{\lambda}(q,t))$  denote the dual basis to $(P_{\lambda}(q,t))$ under the $q,t$ scalar product $\langle \ , \ \rangle_{q,t}$. For a partition $\mu$, let \begin{eqnarray}
		c_{\mu}(q,t) & = &\prod_{x\in \mu}(1-q^{a(x)} \; t^{l(x)+1}) \\ c'_{\mu}(q,t) & = & \prod_{x\in \mu}(1-q^{a(x)+1}\; t^{l(x)}) \\  b_{\mu}(q,t) & = & \dfrac{c_{\mu}(q,t)}{c'_{\mu}(q,t)} 
	\end{eqnarray} 
	
	Then we have \begin{theorem}\cite[Chapter VI (6.19)]{Mac95} 
		\[ Q_{\mu}(q,t) = b_{\mu}(q,t)P_{\mu}(q,t) \]
	\end{theorem}
	
	The \emph{integral form Macdonald polynomials} are  \begin{eqnarray}\label{eqn:Jdefn}
		J_{\mu}(q,t) = c_{\mu}(q,t) P_{\mu}(q,t) = c'_{\mu}(q,t) Q_{\mu}(q,t)
	\end{eqnarray} The monomial coefficients of $J_{\mu}(q,t)$ are in $\mathbb{Z}[q,t]$.

	\subsection{$q,t$-Kostka functions}
	We define two $q,t$-analogues of the Kostka numbers. \begin{definition}\label{defnKi}
		For $\lambda,\mu\in Par$, let $K^{(1)}_{\lambda,\mu}(q,t)$ and $K^{(2)}_{\lambda,\mu}(q,t)$ be defined as the change of basis coefficients: \begin{eqnarray}
			P_{\lambda}(q,t) = \sum_{\mu}K^{(1)}_{\lambda,\mu}(q,t)\; m_{\mu} \\ s_{\lambda} = \sum_{\mu}K^{(2)}_{\lambda,\mu}(q,t)\; P_{\mu}(q,t) 
		\end{eqnarray}
	\end{definition}
	
	In other words, $K^{(1)}_{\lambda,\mu}(q,t) = \langle P_{\lambda}(q,t) , h_\mu \rangle$ and $K^{(2)}_{\lambda,\mu}(q,t) = \langle s_{\lambda} , Q_{\mu} \rangle_{q,t}$ 
	Note that both the change of basis matrices are upper- unitriangular, i.e, for $i\in \{1,2\}$, $K^{(i)}_{\lambda,\mu}(q,t) = 0$ unless $\mu \leq \lambda$ and for all $\lambda \in Par$, $K^{(i)}_{\lambda,\lambda}(q,t)=1$. For $K^{(1)}$ this follows from the definition of Macdonald polynomials and for $K^{(2)}$ this is because \begin{eqnarray}\label{K=K2K1}
		K = K^{(2)}K^{(1)}
	\end{eqnarray} where $K$ is the matrix of Kostka numbers, which is known to be upper-unitriangular.
	
	\subsection{Macdonald's formula for $K^{(1)}_{\lambda,\mu}(q,t)$}
	
	Macdonald gave a formula for $K^{(1)}_{\lambda,\mu}(q,t)$ in \cite{Mac95} which we now describe. 
	For two partitions $\lambda,\mu$ with $\mu_i \leq \lambda_i$ for all $i$, $\lambda/\mu$ is called a horizontal strip if $\lambda'_j-\mu'_j \in \{0,1\}$ for all $j\in \mathbb{Z}_{>0}$. 
	If $\lambda/\mu$ is a horizontal strip, then define \begin{eqnarray}
		\psi_{\lambda/\mu}(q,t)= \prod_{\substack{x=(r,c)\in \mu \\ \lambda_r \neq \mu_r \\ \lambda'_c = \mu'_c  } }\dfrac{(1-q^{a_{\mu}(x)}\; t^{\; l_{\mu}(x)+1})}{(1-q^{a_{\mu}(x)+1}\; t^{\; l_{\mu}(x)})}\dfrac{(1-q^{a_{\lambda}(x)+1}\; t^{\; l_{\lambda}(x)})}{(1-q^{a_{\lambda}(x)}\; t^{\; l_{\lambda}(x)+1})} 
	\end{eqnarray} 
	
	Given a tableau $T$, let $T_{\leq i}$ denote the shape obtained by the boxes with content $1,...,i$. In a tableau $T$, $T_{\leq i}/T_{\leq (i-1)}$ is by definition horizontal strip for each $i$. For a tableau $T$, define \begin{eqnarray}
		\psi_T = \prod_{i \geq 1}\psi_{T_{\leq i}/T_{\leq ( i-1)}}
	\end{eqnarray}   
	
	Then we have \begin{theorem}\cite[Chapter VI, ($7.13^\prime$)]{Mac95} For partitions $\lambda,\mu$,
		\begin{eqnarray}\label{eq:K1=sumpsi}
			K^{(1)}_{\lambda,\mu}(q,t) = \sum_{T\in SSYT(\lambda,\mu)}\psi_T(q,t)
		\end{eqnarray}
	\end{theorem}

	\subsection{Specializations}\label{specializations}
	\begin{enumerate}
		\item\label{KostkaFoulkesDefn} At $q=0$, $K^{(2)}_{\lambda,\mu}(0,t)$ is the Kostka-Foulkes polynomial $K_{\lambda,\mu}(t)$. It is a polynomial in $t$ with non-negative integer coefficients. A combinatorial formula for the Kostka-Foulkes polynomial was given by Lascoux-Schutzenberger.  There is a function $ch : SSYT(\lambda,\mu) \to \mathbb{Z}_{\geq 0}$ called charge, such that \begin{eqnarray}
			K_{\lambda,\mu}(t) = \sum_{T\in SSYT(\lambda,\mu)} t^{ch(T)}
		\end{eqnarray}
		
		\item At $t=1$, $K^{(2)}_{\lambda,\mu}(q,1) = K_{\lambda,\mu}$ for arbitrary $q$.
		
		\item At $q=1$, the matrix $K^{(2)}(1,t) = J(K^{-1})^{tr}J $, where $J$ is the matrix indexed by partitions, with $J_{\lambda,\mu} = \delta_{\lambda,\mu'}$. So the matrix $K^{(2)}$ contains information of both the matrix $K$ and its inverse. A general version is \autoref{Kdual} below. 
		
	\end{enumerate}
	
	\subsection{Duality}
	
	By $(K^{(i)})^{-1}_{\lambda,\mu}$ we mean taking the $\lambda,\mu$th coordinate of the inverse matrix.

	The following lemma says that expanding the Macdonald functions in the Schur basis is roughly equivalent to its inverse problem, expanding the Schur functions in the Macdonald basis.
	
	\begin{lemma}\cite[Chapter VI, ($5.1^{\prime\prime}$)]{Mac95}\label{Kdual}
		\begin{eqnarray}\label{eqnKdual}
			K^{(2)}_{\lambda,\mu}(q,t) = (K^{(2)})^{-1}_{\mu',\lambda'}(t,q)
		\end{eqnarray}
	\end{lemma}
	Since $c_{\lambda}(q,t) = c'_{\lambda'}(t,q)$ \cite[Chapter VI, (8.2)]{Mac95} from \autoref{eqn:Jdefn},\autoref{eqnKdual} and \autoref{defnKi} we get 
	\begin{eqnarray}\label{Jduality}
		\langle J_{\lambda}(t,q),s_\mu \rangle = \langle J_{\lambda'}(q,t),s_{\mu'}\rangle_{q,t} 
	\end{eqnarray}

	\subsection{Haglund's conjecture}\label{sec:HagConj}
	
	In \cite{HHL06} it was shown that the monomial coefficients $\langle J_{\lambda}(q,t),h_{\mu} \rangle$ of $J_{\lambda}(q,t)$ have the following positivity property:	
	\begin{eqnarray}
		\dfrac{\langle J_{\lambda}(q,q^k),h_{\mu} \rangle}{(1-q)^{|\lambda|}} \in \mathbb{Z}_{\geq 0}[q] \quad \forall \ k \in \mathbb{Z}_{\geq 0}
	\end{eqnarray} 
	Haglund made the following conjecture in \cite{Hag10} about the Schur coefficients $\langle J_{\lambda}(q,t),s_{\mu} \rangle$ of $J_{\lambda}(q,t)$:
	\begin{conjecture}[Haglund]\label{hagconj} 
		For partitions $\lambda$ and $\mu$,
		\begin{equation*}
			\dfrac{\langle J_{\lambda}(q,q^k),s_\mu \rangle}{(1-q)^{|\lambda|}} \in \mathbb{Z}_{\geq 0}[q] \quad \forall \ k \in \mathbb{Z}_{\geq 0} \tag{Hag($\lambda,\mu$)} 
		\end{equation*}
	\end{conjecture}
	Meesue Yoo showed this is true in some special cases in \cite{Yoo12}, \cite{Yoo15}. 
	
	Let \begin{eqnarray}\label{eqn:kdefn}
		k_{\lambda,\mu}(q,t) = \langle s_{\lambda}, J_{\mu}(q,t) \rangle_{q,t}
	\end{eqnarray} 
	
	Then \begin{eqnarray}\label{k=K2c'}
		k_{\lambda,\mu}(q,t) = K^{(2)}_{\lambda,\mu}(q,t)c'_{\mu}(q,t) = (K^{(2)})^{-1}_{\mu',\lambda'}(t,q)c'_{\mu}(q,t)
	\end{eqnarray}
	
	Note that $c'_{\mu}(0,t) = 1$, so \begin{eqnarray}
		k_{\lambda,\mu}(0,t) = K_{\lambda,\mu}(t)
	\end{eqnarray} where the right hand side is the Kostka-Foulkes polynomial as in \autoref{specializations}. By \autoref{Jduality} we have the following dual version of Haglund's conjecture.
	\begin{conjecture} 
		For partitions $\lambda,\mu$,
		\begin{equation*}
			\dfrac{k_{\lambda,\mu}(t^k,t)}{(1-t)^{|\lambda|}} \in \mathbb{Z}_{\geq 0}[t] \quad \forall \ k \in \mathbb{Z}_{\geq 0} \tag{Hag$^\prime(\lambda,\mu$)} \label{hagdualconj}
		\end{equation*}
	\end{conjecture}

		\section{When $\lambda$ is a single row or $\mu$ is a single column}
		
		In this section we derive formulas for $K^{(2)}_{(n),\mu}(q,t)$ and $K^{(2)}_{\lambda,1^n}(q,t)$, and use it to show  \ref{hagdualconj} is true in these cases.
		
		We begin by reviewing Cauchy formulas and Principal Specializations.

		\subsection{Cauchy identity}

		For a series of rational functions $E$ we let $\Omega[E] = \sum_{n \geq 0} h_n[E]$. Let $(u_{\lambda})_{\lambda \in \text{Par}}$, $(v_{\lambda})_{\lambda \in \text{Par}}$ be Hall-dual bases, then the Cauchy identity says that \begin{eqnarray}
			\Omega[XY] = \sum_{\lambda \in \text{Par}}u_{\lambda}[X]v_\lambda[Y]
		\end{eqnarray}  
		
		If $(u_{\lambda})_{\lambda \in \text{Par}}$, $(v_{\lambda})_{\lambda \in \text{Par}}$ are $q,t$-dual bases, then by definition of $\langle , \rangle_{q,t}$  $\bigg(u_{\lambda}\bigg[X\dfrac{1-q}{1-t}\bigg]\bigg)_{\lambda \in Par}$, $(v_{\lambda}[X])_{\lambda \in Par}$ are Hall-dual so, \begin{eqnarray}
			\Omega[XY] = \sum_{\lambda \in Par} u_{\lambda}\bigg[X\dfrac{1-q}{1-t}\bigg] v_{\lambda}[Y]
		\end{eqnarray}
		or, by replacing $X$ with $X\dfrac{1-t}{1-q}$ we get the $(q,t)$-Cauchy identity: \begin{eqnarray}\label{eqn:qtcauchy}
			\Omega\bigg[XY\dfrac{1-t}{1-q}\bigg] = \sum_{\lambda \in Par} u_{\lambda}[X] v_{\lambda}[Y]
		\end{eqnarray}
		
		Since $(h_{\lambda})_{\lambda\in Par}$ and $(m_{\lambda})_{\lambda \in Par}$ are Hall-dual, $\bigg(h_{\lambda}\bigg[ X \dfrac{1-t}{1-q} \bigg]\bigg)_{\lambda\in Par}$ is $q,t$-dual to $(m_{\lambda}[X])_{\lambda \in Par}$. 
		
		We give a quick proof of the following well-known result, \cite{Mac95}[Chapter VI: (5.5)] to illustrate the power of plethystic methods. 
		
		\begin{proposition}\label{Q_n}
			\begin{eqnarray}
				Q_{(n)}[X;q,t] = h_{n}\bigg[X \dfrac{1-t}{1-q}\bigg]
			\end{eqnarray}
		\end{proposition}
		
		\begin{proof}
			Cauchy identity gives \begin{eqnarray}\label{handmincauchy}
				\sum_{\lambda \in Par} h_{\lambda}\bigg[X\dfrac{1-t}{1-q}\bigg]m_{\lambda}[Y] = \sum_{\lambda \in Par}Q_{\lambda}[X;q,t]P_{\lambda}[Y;q,t]
			\end{eqnarray} 
			
			Substitute $Y=y$, where $y$ is a variable. Since \begin{eqnarray}
				m_{\lambda}[y]=m_{\lambda}(y,0,0,...)=  \begin{cases}
					y^n ,& \text{ if } \lambda = (n) \text{ for some } n\geq 0 \\ 0 ,& \text{otherwise} 
				\end{cases}
			\end{eqnarray}
			and by uni-triangularity for $P_{\lambda}$, \begin{eqnarray}
				P_{\lambda}[y]=  \begin{cases}
					y^n ,& \text{ if } \lambda = (n) \text{ for some } n\geq 0 \\ 0 ,& \text{otherwise} 
				\end{cases}
			\end{eqnarray}
			
			Hence, by \autoref{handmincauchy} above, \begin{eqnarray}
				\sum_{n \geq 0}h_{(n)}\bigg[X \dfrac{1-t}{1-q} \bigg]y^n = \sum_{n \geq 0} Q_{(n)}[X;q,t]y^n
			\end{eqnarray}
			
			Comparing coefficient of $y^n$ we get the result.
		\end{proof}

		\subsection{Principal Specializations}
		
		We can use the Cauchy identity and plethystic substituion to calculate the first row of various transition matrices. The next lemma and the corollaries illustrate that. We will need Macdonald's evaluation identity stated below.
		
		\begin{theorem}\cite[Chapter VI (6.17)]{Mac95} \label{Maceval}
			Let $z$ be a complex variable. Then \begin{eqnarray}
				P_{\lambda}\bigg[\dfrac{1-z}{1-t};q,t\bigg] = \prod_{s \in \lambda} \dfrac{t^{l'(s)}-q^{a'(s)}z}{1-q^{a(s)}t^{l(s)+1}} \\ Q_{\lambda}\bigg[\dfrac{1-z}{1-t};q,t\bigg] = \prod_{s \in \lambda} \dfrac{t^{l'(s)}-q^{a'(s)}z}{1-q^{a(s)+1}t^{l(s)}}
			\end{eqnarray} 
		\end{theorem}
		
		\subsection{When $\lambda$ is a row or $\mu$ is a column}
		We now calculate $k_{\lambda,\mu}(q,t)$ in these special cases.
		Substituting $Y$ with $\dfrac{1-z}{1-t}$ in the Cauchy identity \autoref{eqn:qtcauchy} and using \autoref{Maceval} we get the following lemma.
		
		\begin{lemma}\label{h_nlemma}
			Let $z$ be a complex variable. For a pair $(u_{\lambda})_{\lambda \in \text{Par}}$ and $(v_{\lambda})_{\lambda \in \text{Par}}$ of Hall-dual bases of $\Lambda$ we have
			\begin{eqnarray}
				h_n\bigg[X\dfrac{1-z}{1-q}\bigg] = \sum_{\lambda \, \vdash n} v_{\lambda}\bigg[ \dfrac{1-z}{1-t}\bigg] u_{\lambda}[X]
			\end{eqnarray}
			In particular, \begin{eqnarray}\label{h_nPlemma}
				h_n\bigg[X\dfrac{1-z}{1-q}\bigg] = \sum_{\lambda \, \vdash n} \prod_{x \in \lambda}\dfrac{t^{l'(x)}-q^{a'(x)}z}{1-q^{a(x)+1}t^{l(x)}}P_{\lambda}[X;q,t]
			\end{eqnarray}
			
			At $t=q$, \begin{eqnarray}\label{h_nslemma}
				h_n\bigg[X\dfrac{1-z}{1-q}\bigg] = \sum_{\lambda \, \vdash n}q^{n(\lambda)}\prod_{x \in \lambda}\dfrac{1-q^{c(x)}z}{1-q^{h(x)}}s_{\lambda}[X]
			\end{eqnarray}
		\end{lemma}

		\begin{proposition}\label{k_n_mu}
			\begin{eqnarray}\label{eqn:k_n_mu}
				k_{(n),\mu}(q,t) = t^{n(\mu)}\prod_{x \in \mu}(1-q^{a'(x)+1}t^{-l'(x)})
			\end{eqnarray}  
			
		\end{proposition}

		\begin{proof}
			By substituting $z=q$ in \autoref{h_nPlemma} and using \autoref{eqn:Jdefn}, \autoref{eqn:kdefn} and the fact that $s_{(n)} = h_n$.	
		\end{proof}
		We can use Lemma \autoref{h_nlemma} to get another proof of the following result of \cite{Yoo12} in the dual formalism.  
		
		\begin{proposition}\label{klambda1nprop}
			\begin{eqnarray}
				k_{\lambda,1^n}(q,t) &=& \dfrac{t^{n(\lambda')}[n]_t!}{\prod_{x \in \lambda}[h(x)]_t} \prod_{x \in \lambda}(1-t^{-c(x)}q) \\ & = &  K_{\lambda,1^n}(t) \prod_{x \in \lambda} (1-t^{-c(x)}q) 
			\end{eqnarray} where $[j]_t = \dfrac{1-t^j}{1-t}$ for $j\in \mathbb{Z}_{\geq 1}$ and $[n]_t! = [n]_t[n-1]_t\ldots[1]_t$ 
		\end{proposition}
		
		\begin{proof}
			Substituting $z=t$ in \autoref{h_nslemma} we get   \begin{eqnarray}
				h_n\bigg[X\dfrac{1-t}{1-q}\bigg] = \sum_{\lambda \vdash n} \prod_{x \in \lambda} \dfrac{q^{l'(x)}-q^{a'(x)}t}{1-q^{h(x)}}s_{\lambda}[X]
			\end{eqnarray}
			By \autoref{Q_n} the left hand side is $Q_n(q,t)$, and so the product in front of $s_{\lambda}[X]$ on the right hand side is $\langle Q_n(q,t),s_{\lambda} \rangle$.\newline Observe that $c'_{(n)}(q,t) = (q;q)_n$ where by $(a;x)_n$ we mean the product $\prod_{i=0}^{n-1}(1-ax^{i})$. Since $c'_{(n)}Q_n = J_n$, \autoref{Jduality} gives: \begin{eqnarray}
				k_{\lambda',1^n}(q,t) = (t;t)_n \prod_{x \in \lambda} \dfrac{t^{l'(x)}-t^{a'(x)}q}{1-t^{h(x)}} 
			\end{eqnarray}
			Since conjugation interchanges coleg lengths and coarm lengths, and the set of hook lengths remains unchanged, \begin{eqnarray}
				k_{\lambda,1^n}(q,t) & = & (t;t)_n \prod_{x \in \lambda} \dfrac{t^{a'(x)}-t^{l'(x)}q}{1-t^{h(x)}} \\ & = & \dfrac{t^{n(\lambda')}(t;t)_n}{\prod_{x \in \lambda}(1-t^{h(x)})} \prod_{x \in \lambda}(1-t^{-c(x)}q) 
			\end{eqnarray}
			Since $k_{\lambda,\mu}(0,t) = K_{\lambda,\mu}(t)$, the result follows.
		\end{proof}

		\begin{corollary}\label{hagforroworcol}
			\emph{\ref{hagdualconj}} holds when $\lambda$ is a row or when $\mu$ is a column. 
		\end{corollary}
		\begin{proof}
			
			When $\lambda = (n)$ and $\mu$ is arbitrary, we have from \autoref{eqn:k_n_mu} \begin{eqnarray}
				k_{(n),\mu}(t^k,t) = t^{n(\mu)}\prod_{x \in \mu}(1-t^{k(a'(x)+1)-l'(x)}) \quad \forall \ k \in \mathbb{Z}_{\geq 0}
			\end{eqnarray}  
			Since the quantity $k(a'(x)+1)-l'(x)$ decreases by $1$ down a column and it is non-negative on the top row, if it is negative for some box then it must also attain the value $0$ for some box above it. Hence \begin{eqnarray}
				\dfrac{k_{(n),\mu}(t^k,t)}{(1-t)^n} = \begin{cases}
					t^{n(\mu)} \prod_{x \in \mu} [k(a'(x)+1)-l'(x)]_t & \text{ if } l(\mu) \leq k , \\ \quad \quad \quad 0 & \text{otherwise}
				\end{cases}  
			\end{eqnarray} where $[m]_t = 1+t+...+t^{m-1}$. In both cases $\dfrac{k_{(n),\mu}(t^k,t)}{(1-t)^n} \in \mathbb{Z}_{\geq 0}[t]$.
			
			Next, when $\mu = 1^n$ and $\lambda$ is arbitrary, we get \begin{eqnarray}
				k_{\lambda,1^n}(t^k,t) = K_{\lambda,1^n}(t) \prod_{x \in \lambda}(1-t^{k-c(x)})
			\end{eqnarray}
			Along a row the quantity $k-c(x)$ decreases by $1$ as one moves from left to right, and for the first box in the row (in the first column) $k-c(x)\geq 0$. So if  $k-c(x)<0$  for some $x\in \lambda$ then in that row there is a box $y$ such that $k-c(y)=0$, implying $k_{\lambda,1^n}(t^k,t)=0$. Otherwise $k-c(x)\geq 0$ for all $x\in \lambda$, and hence \begin{eqnarray}
				\dfrac{k_{\lambda,1^n}(t^k,t)}{(1-t)^n} = \begin{cases}
					K_{\lambda,1^n}(t) \prod_{x \in \lambda}[k-c(x)]_t & \text{if } \lambda_1 \leq k \\ \quad\quad\quad 0 & \text{otherwise} 
				\end{cases} 
			\end{eqnarray}
			Since $K_{\lambda,\mu}(t)\in \mathbb{Z}_{\geq 0}[t]$, $\dfrac{k_{\lambda,1^n}(t^k,t)}{(1-t)^n} \in \mathbb{Z}_{\geq 0}[t]$.
			
		\end{proof}

		\section{Multiplication}
		
		We state a general lemma on posets that will be useful in this section. 
		
		\begin{lemma}\label{posetlemma}
			Let $M$ be an upper uni-triangular matrix indexed by a partially ordered set $P$, i.e, for all $a\in P$,   $M_{a,a}=1$ and $M_{a,b}=0$ unless $b\leq a$. For an interval $I$ in $P$ let $M|_{I}=(M_{a,b})_{a,b\in I}$ denote the matrix restricted to the interval $I$. Then for any interval $J$ containing $a,b$, we have $(M^{-1})_{a,b} = (M|_{J})^{-1}_{a,b}$.  
		\end{lemma}
		
		\begin{proof}
			Let $N=M^{-1}$. The lemma follows since one can simply write down the solution to the equations \begin{eqnarray}
				\delta_{a,b} = \sum_{a\leq c \leq b}N_{a,c}M_{c,b}
			\end{eqnarray}
			recursively, starting from $b=a$.
		\end{proof}

		Let $\lambda^{1}, \lambda^2$ be two partitions such that the least non-zero part of $\lambda^1$ is atleast the largest part of $\lambda^2$. In this case by $(\lambda^1, \lambda^2)$ we mean the partition $(\lambda^1_1,...,\lambda^1_{l(\lambda^1)},\lambda^2_1,...)$.
		
		\begin{lemma}\label{rowmult}
			Let $\lambda = (\lambda^1,\lambda^2)$ and $\mu = (\mu^1,\mu^2)$ where $\lambda^j \geq\mu^j $ (in particular, $|\lambda^j|=|\mu^j|$) for $j=1,2$  and $l(\lambda^1)=l(\mu^1)$. 
			
			Then for $i=1,2$, \begin{eqnarray}
				(K^{(i)})^{\pm 1}_{\lambda,\mu} = (K^{(i)})^{\pm 1}_{\lambda^1,\mu^1}(K^{(i)})^{\pm 1}_{\lambda^2,\mu^2}
			\end{eqnarray}
			
			
		\end{lemma}
		
		\begin{proof}
			Let $T \in SSYT(\lambda,\mu)$. Since $l(\lambda^1)=l(\mu^1)$ and $|\lambda^1|=|\mu^1|$, in the first $l(\lambda^1)$ rows of $T$ the content must be $\mu^1$. We can break the tableau $T$ into two pieces; let $T^1$ comprise the first $l(\lambda^1)$ rows, and $T^2$ be the tableau formed by the remaining rows of $T$ in which we subtract $l(\lambda^1)$ from each entry. It is clear we get a bijection \begin{eqnarray}
				SSYT(\lambda,\mu) & \rightarrow & SSYT(\lambda^1,\mu^1) \times SSYT(\lambda^2,\mu^2) \nonumber \\ T & \mapsto & (T^1,T^2) \nonumber
			\end{eqnarray} 
			
			In particular, $K_{\lambda,\mu} = K_{\lambda^1,\mu^1}K_{\lambda^2,\mu^2}$.

			For $1 \leq i \leq l(\lambda^1)$ the horizontal strips $T_{\leq i}/T_{\leq i-1}$ and $T^1_{\leq i}/T^1_{\leq i-1}$ are the same, so $\psi_{T_{\leq i}/T_{\leq i-1}} = \psi_{T^1_{\leq i}/T^1_{\leq i-1}}$. 
			
			For $i > l(\lambda^1)$ the horizontal strips $T_{\leq i}/T_{\leq i-1}$ differ from $T^2_{\leq i-l(\lambda^1)}/T^2_{\leq i-l(\lambda^1)-1}$ in the first $l(\lambda^1)$ rows, but none of those boxes contribute to $\psi_{T_{\leq i}/T_{\leq i-1}}$, and the contribution from the remaining boxes are same since the arms and legs are the same. So $\psi_{T_{\leq i}/T_{\leq i-1}} = \psi_{T^2_{\leq i-l(\lambda^1)}/T^2_{\leq i-l(\lambda^1)-1}}$ 
			
			Therefore the above bijection preserves $\psi$: \begin{eqnarray}
				\psi_T = \psi_{T^1}\psi_{T^2}
			\end{eqnarray}
			
			By \autoref{eq:K1=sumpsi}  we get \begin{eqnarray}
				K^{(1)}_{\lambda,\mu} = K^{(1)}_{\lambda^1,\mu^1}K^{(1)}_{\lambda^2,\mu^2}
			\end{eqnarray}

			Let $\lambda \geq \nu \geq \mu$. Letting $\nu^1 = (\nu_1,...,\nu_{l(\lambda^1)})$ and $\nu^2 = (\nu_{l(\lambda^1)+1},...)$ we have that $\lambda^i \geq \nu^i \geq \mu^i$ for $i=1,2$ and $l(\lambda^1) = l(\nu^1) = l(\mu^1)$. So the interval $[\lambda,\mu]$ in the dominance order becomes $[\lambda,\mu] = [\lambda^1,\mu^1] \times [\lambda^2,\mu^2]$. The matrix $K^{(1)}$ restricted to $[\lambda,\mu]$ is then the tensor product: \begin{eqnarray}
				K^{(1)}|_{[\lambda,\mu]} = K^{(1)}|_{[\lambda^1,\mu^1]} \otimes K^{(1)}|_{[\lambda^2,\mu^2]}   
			\end{eqnarray} In particular, $ 
			K|_{[\lambda,\mu]} = K|_{[\lambda^1,\mu^1]} \otimes K|_{[\lambda^2,\mu^2]}$.
			
			Since inverse of the tensor product of two matrices is the tensor product of their inverses we get 
			\begin{eqnarray}
				(K^{(1)})^{-1}|_{[\lambda,\mu]} = (K^{(1)})^{-1}|_{[\lambda^1,\mu^1]} \otimes (K^{(1)})^{-1}|_{[\lambda^2,\mu^2]}   
			\end{eqnarray}
			
			By \autoref{K=K2K1}

			\begin{eqnarray}
				K^{(2)}|_{[\lambda,\mu]} = K^{(2)}|_{[\lambda^1,\mu^1]} \otimes K^{(2)}|_{[\lambda^2,\mu^2]}  
			\end{eqnarray}
			
			Once again, taking inverses we get \begin{eqnarray}
				(K^{(2)})^{-1}|_{[\lambda,\mu]} = (K^{(2)})^{-1}|_{[\lambda^1,\mu^1]} \otimes (K^{(2)})^{-1}|_{[\lambda^2,\mu^2]}  
			\end{eqnarray}
			

			By \autoref{posetlemma}, we conclude for $i=1,2$ \begin{eqnarray}
				(K^{(i)})^{\pm 1}_{\lambda,\mu} = (K^{(i)})^{\pm 1}_{\lambda^1,\mu^1}(K^{(i)})^{\pm 1}_{\lambda^2,\mu^2}
			\end{eqnarray}
			
		\end{proof}

		Let $\lambda^1, \lambda^2$ be two partitions such that the last non-zero column length of $\lambda^{1}$ is greater equal to the first column length of $\lambda^{2}$. Then by $\lambda^1+\lambda^2$ we denote the partition $(\lambda^1_1+\lambda^2_1,\lambda^1_2+\lambda^2_2,\ldots)$. Note that if $\lambda = \lambda^1+\lambda^2$ then $\lambda' = (\lambda^{1'},\lambda^{2'})$ in the notation of \autoref{rowmult}. By $\lambda-\lambda^1$ we mean the partition $\lambda^2$. 		
		
		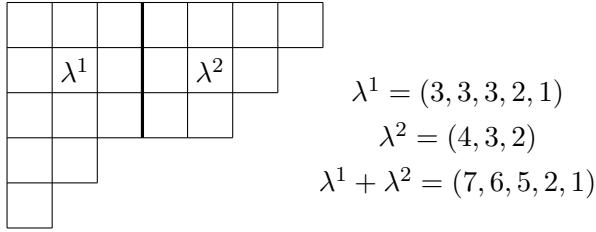
\begin{figure}[h]
			\centering
			\begin{tikzpicture}[yscale=-0.6, xscale=0.6]
				\draw (0,0) -- (7,0);
				\draw (0,1) -- (7,1);
				\draw (0,2) -- (6,2);
				\draw (0,3) -- (5,3);
				\draw (0,4) -- (2,4);
				\draw (0,5) -- (1,5);
				\draw (0,0) -- (0,5);
				\draw (1,0) -- (1,5);
				\draw (2,0) -- (2,4);
				\draw[very thick] (3,0) -- (3,3);
				\draw (4,0) -- (4,3);
				\draw (5,0) -- (5,3);
				\draw (6,0) -- (6,2);
				\draw (7,0) -- (7,1);
				
				\draw (1.5,1.5) node{$\lambda^1$};
				\draw (4.5,1.5) node{$\lambda^2$};
				
				\draw (10,2) node{$\lambda^1=(3,3,3,2,1)$};
				\draw (10,3) node{$\lambda^2=(4,3,2)$};
				\draw (10,4) node{$\lambda^1+\lambda^2=(7,6,5,2,1)$};
			\end{tikzpicture}	
			
			\caption{Example of $\lambda^1+\lambda^2$}
		\end{figure}
		
		\begin{lemma}\label{colmult}
			Let $\lambda,\mu \in Par$ such that $\lambda = \lambda^1+\lambda^2$ and $\mu = \mu^1+\mu^2$ with $\lambda^1 \geq \mu^1$, $\lambda^2 \geq \mu^2$ and $\lambda^1_1 = \mu^1_1$. Then for $i=1,2$, \begin{eqnarray}
				(K^{(i)})^{\pm 1}_{\lambda,\mu} = (K^{(i)})^{\pm 1}_{\lambda^1,\mu^1} (K^{(i)})^{\pm 1}_{\lambda^2,\mu^2}
			\end{eqnarray} 
			
			
		\end{lemma}
		
		\begin{proof}
			
			Note that $\lambda' = (\lambda^{1'},\lambda^{2'})$ and $\mu' = (\mu^{1'},\mu^{2'})$, $\mu^{1'} \geq \lambda^{1'}$ and $l(\mu^{1'}) = \mu^1_1 = \lambda^1_1 = l(\lambda^{1'})$. Therefore we can apply \autoref{rowmult} to $\mu',\lambda'$. By \autoref{Kdual} 	
			\begin{eqnarray}
				K^{(2)}_{\lambda,\mu}(q,t) & = & (K^{(2)})^{-1}_{\mu',\lambda'}(t,q) \\ & = & (K^{(2)})^{-1}_{\mu^{1'},\lambda^{1'}}(t,q)(K^{(2)})^{-1}_{\mu^{2'},\lambda^{2'}}(t,q) \\ & = & K^{(2)}_{\lambda^1,\mu^1}(q,t)K^{(2)}_{\lambda^2,\mu^2}(q,t) 
			\end{eqnarray}
			
			In particular, \begin{eqnarray}
				K_{\lambda,\mu} = K_{\lambda^1,\mu^1}K_{\lambda^2,\mu^2}	
			\end{eqnarray}
			
			If a partition $\gamma$ satisfies $\lambda \geq \gamma \geq \mu$ then $\mu' \geq \gamma' \geq \lambda'$. By earlier argument, $[\mu',\lambda']=[\mu^{1'},\lambda^{1'}]\times [\mu^{2'},\lambda^{2'}]$. Then $[\lambda,\mu] = [\lambda^1,\mu^1]+[\lambda^2,\mu^2]$ and any $\gamma^1 \in [\lambda^1,\mu^1]$ satisfies $\lambda^1_1 = \gamma^1_1 = \mu^1_1$. Thus \begin{eqnarray}
				K^{(2)}|_{[\lambda,\mu]} = K^{(2)}|_{[\lambda^1,\mu^1]} \otimes K^{(2)}|_{[\lambda^2,\mu^2]}  
			\end{eqnarray}
			As before, we take inverses and use \autoref{K=K2K1}, to prove the statement.   
		\end{proof}

		Let $\lambda^1,\lambda^2,\mu^1,\mu^2$ be as in \autoref{rowmult}. Note that since $l(\lambda^1)= l(\mu^1)$ and $\lambda^1 \geq \mu^1$, we have $\mu^1_{l(\lambda^1)} \geq \lambda^1_{l(\lambda^1)}$. Thus both $\lambda^1$ and $\mu^1$ contain the rectangle $R$ of row length $\lambda^1_{l(\lambda^1)}$ and column length $l(\lambda^1)$. 
		
		\begin{corollary}\label{cor:rectangleinsert}
			Let $\lambda^1,\lambda^2,\mu^1,\mu^2$ be as in \autoref{rowmult}, and $R$ be the rectangular partition of row length $\lambda^1_{l(\lambda^1)}$ and column length $l(\lambda^1)$. If \emph{\ref{hagdualconj}} is true for $\lambda^1-R, \mu^1-R$ and for $\lambda^2,\mu^2$ then it is true for $\lambda,\mu$. 
			
		\end{corollary}
		
		\begin{proof}
			
			
			By applying \autoref{colmult} to $\lambda^1 = R + (\lambda^1-R)$ and $\mu^1 = R + (\mu^1-R)$ we get for $i=1,2$,  \begin{eqnarray}
				K^{(i)}_{\lambda^1,\mu^1} = K^{(i)}_{R,R}K^{(i)}_{\lambda^1-R,\mu^1-R} = K^{(i)}_{\lambda^1-R,\mu^1-R} 
			\end{eqnarray}
			
			So, by \autoref{rowmult} \begin{eqnarray}\label{eqn:Kiprod}
				K^{(i)}_{\lambda,\mu} = K^{(i)}_{\lambda^1,\mu^1}K^{(i)}_{\lambda^2,\mu^2}= K^{(i)}_{\lambda^1-R,\mu^1-R}K^{(i)}_{\lambda^2,\mu^2}
			\end{eqnarray}
			
			Since the arms and legs of boxes in $\mu^1-R$ have no intersection with those of $\mu^2$, it follows that \begin{eqnarray}
				\dfrac{c'_{\mu}}{c'_{\mu^1-R}c'_{\mu^2}} = \prod_{x \in R}(1-q^{a_{\mu}(x)+1}t^{l_{\mu}(x)})
			\end{eqnarray}
			
			From \autoref{k=K2c'}, we get \begin{eqnarray}\label{kmultcor}
				k_{\lambda,\mu}(q,t) = k_{\lambda^1-R,\mu^1-R}(q,t)\; k_{\lambda^2,\mu^2}(q,t)\; \prod_{x \in R}(1-q^{a_{\mu}(x)+1}t^{l_{\mu}(x)})
			\end{eqnarray} 
			
			So, \begin{eqnarray}
				\dfrac{k_{\lambda,\mu}(t^k,t)}{(1-t)^{|\lambda|}} = \dfrac{k_{\lambda^1-R,\mu^1-R}(t^k,t)}{(1-t)^{|\lambda^1|-|R|}}\; \dfrac{k_{\lambda^2,\mu^2}(t^k,t)}{(1-t)^{|\lambda^2|}}\; \dfrac{\prod_{x \in R}(1-t^{k(a_{\mu}(x)+1)+l_{\mu}(x)})}{(1-t)^{|R|}}
			\end{eqnarray}  
			The last term on the right hand side is a product of $t$-numbers and hence in $\mathbb{Z}_{\geq 0}[t]$.
			
		\end{proof}
		
		
		As particular cases we get the next corollary \begin{corollary}\label{kplusrectangle}
			If \emph{\ref{hagdualconj}} is true for $\lambda,\mu$ then it is true for $(R,\lambda),(R,\mu)$ and for $(S+\lambda),(S+\mu)$ where $R,S$  are rectangular partitions, with row-length of $R$ atleast as big as $\lambda_1$ and column length of $S$ atleast as big as $l(\mu)$
		\end{corollary}
		
		\begin{remark}
			We deduce \cite[Theorem 3.1.1]{Yoo12} in the dual setup by \autoref{kmultcor} and using \autoref{k_n_mu}. In particular this proves \emph{\ref{hagdualconj}} holds for $\lambda = (a+k,b-k)$ and $\mu =(a,b)$. 
			Similarly, we obtain \cite[Proposition 3.4]{Yoo15} in the dual setup by applying \autoref{kmultcor} and using \autoref{klambda1nprop}.  
		\end{remark}

		\begin{example} We demonstrate the two principles of \autoref{rowmult} and \autoref{colmult} in this example (see \autoref{eqn:Kiprod})
			\begin{eqnarray*}
				K^{(2)}_{533,44111}(q,t) = K^{(2)}_{53,44}(q,t)K^{(2)}_{3,111}(q,t) =  K^{(2)}_{2,11}(q,t)K^{(2)}_{3,111}(q,t)
			\end{eqnarray*}
		\end{example}
		
		\section{Complementation}
		
		In this section we switch to type $GL_n$ Macdonald polynomials and then come back to symmetric functions. The reference for this section is \cite{Mac95} Chapter VI section 9. 
		
		Let $n \in \mathbb{Z}_{>0}$ be fixed and let $q$ and $t$ be two complex numbers, with $t=q^k$ for some fixed $k \in \mathbb{Z}_{\geq 0}$. 
		
		Let $\mathcal{P} = \mathbb{Z}^n$, and $\mathcal{P}^+ = \{ \alpha \in \mathcal{P} : \alpha_1 \geq \alpha_2 \geq ... \geq \alpha_n \}$. For $\alpha=(\alpha,\ldots,\alpha_n)\in \mathbb{Z}^n$ let $x^\alpha = (x_1^{\alpha_1},\ldots,x_n^{\alpha_n})$. Let $W=S_n$. By $\mathbb{C}[\mathcal{P}]$ we denote the ring of Laurent polynomials $\mathbb{C}[x_1^{\pm 1},x_2^{\pm 1},...,x_n^{\pm 1}]$. There is a ring homomorphism $\Lambda \to \mathbb{C}[\mathcal{P}]^W$ defined by sending $x_{n+i} \mapsto 0$ for $i\geq 1$. We denote the image of a symmetric function $f$ by $f(X_n)$. For any $\lambda \in \mathcal{P}^+$ (not necessarily with non-negative coordinates) we can define $P_{\lambda}(x;q,t) \in \mathbb{C}[\mathcal{P}]^W$ by the relations: \begin{eqnarray}\label{Plam+1^n}
			P_{\lambda}(x;q,t) & = & P_{\lambda}(X_n;q,t) \quad \text{ if } \ \lambda_n \geq 0 \\
			P_{\lambda+(1^n)}(x;q,t) & = & (x_1...x_n)P_{\lambda}(x;q,t) \quad \forall \ \lambda \in \mathcal{P}^+
		\end{eqnarray} 
		We also define for all $\lambda \in \mathcal{P}^+$, $m_{\lambda}(x) = P_{\lambda}(x;q,1)$ and $s_{\lambda}(x) = P_{\lambda}(x;q,q)$.
		
		
		Also, $\mathcal{P}^+$ has the dominance partial order $\geq$ defined by \begin{eqnarray}
			\lambda \geq \mu \ \text{if and only if} \ \lambda-\mu \in \mathbb{Z}_{\geq 0}\{e_i-e_j:1 \leq i<j \leq n\}
		\end{eqnarray}
		where $e_i$ denotes the standard basis vector $(0,\ldots,1,\ldots,0)$ with $1$ in $i$-th position and $0$ everywhere else.
		
		Let \ $\bar{}: \mathbb{C}[\mathcal{P}] \to \mathbb{C}[\mathcal{P}]$ be the $\mathbb{C}$-algebra involution defined by $x_i \mapsto x_i^{-1}$ for $i\in \{ 1,...,n \}$. Then there is a scalar product on $\mathbb{C}[\mathcal{P}]$ defined by \begin{eqnarray}
			\langle f , g \rangle' = \dfrac{1}{n!} \text{ct}(f\overline{g}\Delta)
		\end{eqnarray}
		where $\text{ct}$ denotes the constant term map on the Laurent polynomial ring, i.e, ct($f$)=coefficient of $x^0$ in $f$ and \begin{eqnarray}
			\Delta = \prod_{i \neq j}\prod_{r=0}^{k-1}(1-q^r\dfrac{x_i}{x_j})
		\end{eqnarray}
		
		The next theorem characterizes the $P_{\lambda}(x;q,t) : \lambda \in \mathcal{P}^+$.
		
		\begin{theorem}
			The $P_{\lambda}(x;q,t): \lambda \in \mathcal{P}^+$ is the unique family of elements in $\mathbb{C}[\mathcal{P}]^W$ satisfying \begin{eqnarray}
				P_{\lambda}(x;q,t) = m_{\lambda}(x;q,t) + \sum_{\mu < \lambda} a_{\lambda,\mu}(q,t) m_{\mu}(x) \\  \ \langle P_{\lambda}(x;q,t) , P_{\mu}(x;q,t) \rangle' = 0 \ \text{for } \ \lambda \neq \mu 
			\end{eqnarray} 
			for some $a_{\lambda,\mu}(q,t) \in \mathbb{C}(q,t) \subset \mathbb{C}$
		\end{theorem}
		
		\begin{proof}
			If $\lambda\in Par$ with $l(\lambda) \leq n$ then $P_{\lambda}(x;q,t)$ is a homogeneous polynomial of degree $|\lambda|$. By \autoref{Plam+1^n}, $P_{\lambda}(X_n;q,t)$ is a homogeneous Laurent polynomial with degree $\lambda_1+\ldots+\lambda_n$ for all $\lambda \in \mathcal{P}^+$ . By comparing degrees it is clear that if $\lambda_1+\ldots+\lambda_n \neq \mu_1+\ldots+\mu_n$ then $\langle P_{\lambda}(x;q,t) , P_{\mu}(x;q,t) \rangle' = 0$. By \autoref{Plam+1^n} and the definition of $\langle \ , \ \rangle'$ it is enough to check orthogonality of the set $\{P_{\lambda}(X_n;q,t) : \lambda \in \mathcal{P}^+ \text{ \ such that} \ \lambda_n \geq 0 \}$. This is shown in \cite{Mac95}[Chapter VI, (9.5)]. Since $\{m_{\lambda}(x): \lambda \in \mathcal{P}^+\}$ is a basis of $\mathbb{C}[\mathcal{P}]^W$ so is $\{ P_{\lambda}(x;q,t): \lambda\in \mathcal{P}^+ \}$.
			
			Any two bases of $\mathbb{C}[\mathcal{P}]^W$ satisfying the above conditions will be related by a triangular orthogonal matrix and any such matrix is necessarily diagonal, so the uniqueness follows.
			
		\end{proof}
		
		Let $w_0 \in S_n$ be the permutation that sends $i \mapsto (n+1-i)$ for $i\in \{ 1,...,n \}$. Let $\phi: \mathbb{C}[\mathcal{P}] \to \mathbb{C}[\mathcal{P}]$ be the $\mathbb{C}$-algebra homomorphism defined by $\phi(x^{\alpha}) = x^{-w_0\alpha}$. So $\phi(f)(x) = w_0 \overline{f}(x)$
		\begin{proposition}
			For $\lambda \in \mathcal{P}^+$, $\phi(P_{\lambda}(x;q,t)) = P_{-w_0\lambda}(x;q,t)$. In particular, $\phi(s_{\lambda}(x)) = s_{-w_0\lambda}(x)$ 
		\end{proposition}  
		
		\begin{proof}
			Since $\phi(m_{\mu}(x)) = m_{-w_0\mu}(x) \quad \forall \mu \in \mathcal{P}^+ $, we have
			\begin{eqnarray}
				\phi(P_{\lambda}(x;q,t)) = m_{-w_0\lambda}(x)+ \sum_{\mu < \lambda}a_{\lambda,\mu}(q,t) m_{-w_0\mu}(x)
			\end{eqnarray}
			Note that $\mu \leq \lambda \iff -w_0\mu \leq -w_0\lambda$, so $\phi(P_\lambda)$ satisfies the unitriangularity property. To prove $\phi(P_{\lambda}(x;q,t)) =  P_{-w_0\lambda}(x;q,t) $, we only need to prove for $\lambda \neq \mu$, $\langle \phi(P_{\lambda}(x;q,t)), \phi(P_{\mu}(x;q,t))\rangle'  = 0$. This follows from the fact that the constant term of a Laurent polynomial is $S_n$ invariant, $w_0$ is a homomorphism and that $\Delta$ is $S_n$-invariant. 

		\end{proof}

		\begin{corollary}\label{complement}
			Let $\lambda, \mu \subset (m^n)$ be partitions. Let $\lambda^c = (m-\lambda_{n},m-\lambda_{n-1},...,m-\lambda_1)$ and $\mu^c = (m-\mu_{n},m-\mu_{n-1},...,m-\mu_1)$. Then for $i=1,2$
			\begin{eqnarray}
				(K^{(i)})^{\pm 1}_{\lambda,\mu} = (K^{(i)})^{\pm 1}_{\lambda^c,\mu^c} 
			\end{eqnarray}  
		\end{corollary} 
		
		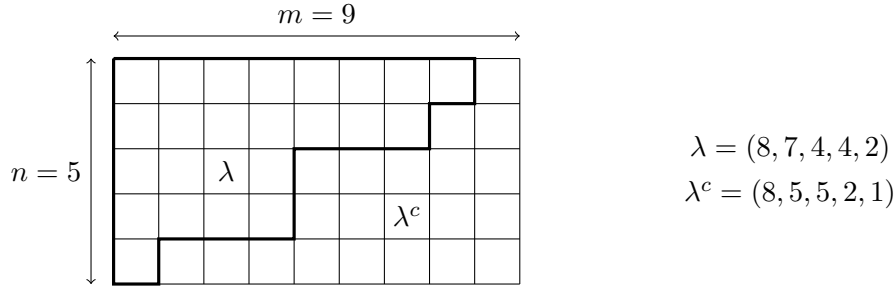
\begin{figure}[h]\label{fig:lamc}
			\begin{tikzpicture}[yscale=-0.6, xscale=0.6]
				\draw[very thick] (0,0)--(8,0)--(8,1)--(7,1)--(7,2)--(4,2)--(4,4)--(1,4)--(1,5)--(0,5)--(0,0);
				
				\draw (0,0)--(9,0);
				\draw (0,1)--(9,1);
				\draw (0,2)--(9,2);
				\draw (0,3)--(9,3);
				\draw (0,4)--(9,4);
				\draw (0,5)--(9,5);
				
				\draw (0,0)--(0,5);
				\draw (1,0)--(1,5);
				\draw (2,0)--(2,5);
				\draw (3,0)--(3,5);
				\draw (4,0)--(4,5);
				\draw (5,0)--(5,5);
				\draw (6,0)--(6,5);
				\draw (7,0)--(7,5);
				\draw (8,0)--(8,5);
				\draw (9,0)--(9,5);
				
				\draw (2.5,2.5) node{$\lambda$};
				\draw (6.5,3.5) node{$\lambda^c$};
				
				\draw[<->] (0,-.5)--(9,-.5);
				\draw (4.5,-1) node{$m = 9$};
				
				\draw[<->] (-.5,0)--(-.5,5);
				\draw (-1.5,2.5) node{$n = 5$};

				\draw (15,2) node{$\lambda = (8,7,4,4,2)$};
				\draw (15,3) node{$\lambda^c = (8,5,5,2,1)$};
			\end{tikzpicture}	
			
			\caption{Example of $\lambda,\lambda^c$ when $m=9,n=5$}
			
		\end{figure}
		
		\begin{proof}
			Let $w_0 \in S_n$ be as before. First note that if $\lambda \subset (m^n)$ is a partition then $\lambda^c = (m-\lambda_{n},...,m-\lambda_1)$ is also a partition. So it makes sense to talk about $K^{(i)}_{\lambda^c,\mu^c}(q,t)$ in this situation. 
			
			We have in the ring $\Lambda$: \begin{eqnarray}
				P_{\lambda}(q,t) = \sum_{\nu \leq \lambda} K^{(1)}_{\lambda,\nu}(q,t) \; m_\nu 
			\end{eqnarray}
			Specializing to $n$ variables we have \begin{eqnarray}
				P_{\lambda}(X_n;q,t) = \sum_{\nu \leq \lambda} K^{(1)}_{\lambda,\nu}(q,t) \; m_{\nu}(X_n)
			\end{eqnarray}
			On the right hand side, if $l(\nu)>n$ then the corresponding term vanishes. Now applying the map $\phi$ to this equation we get \begin{eqnarray}
				P_{-w_0\lambda}(X_n;q,t) = \sum_{\nu \leq \lambda} K^{(1)}_{\lambda,\nu}(q,t) m_{-w_0\nu}(X_n) 
			\end{eqnarray} 
			Multiplying by $(x_1...x_n)^m$ and using \autoref{Plam+1^n} we get \begin{eqnarray}\label{eqnPcomplementX_ntomcomplementX_n}
				P_{(m^n)-w_0\lambda}(X_n;q,t) = \sum_{\nu \leq \lambda} K^{(1)}_{\lambda,\nu}(q,t) m_{(m^n)-w_0\nu}(X_n) 
			\end{eqnarray}
			
			Again, consider the equation in $\Lambda$ \begin{eqnarray}
				P_{\lambda^c}(q,t) = \sum_{\gamma \leq \lambda^c} K^{(1)}_{\lambda^c,\gamma}(q,t)m_{\gamma}
			\end{eqnarray} and specialize to $n$ variables.
			We get that the coefficient of $m_{\mu^c}(X_n)$ is $K^{(1)}_{\lambda^c,\mu^c}(q,t)$. Here we have used the fact that $m_{\mu^c}(X_n) \neq 0$ since $\mu \subset (m^n)$. Comparing with \autoref{eqnPcomplementX_ntomcomplementX_n} we get the desired statement for $K^{(1)}$. By considering \begin{eqnarray}
				m_{\lambda} = \sum_{\nu \leq \lambda}(K^{(1)})^{-1}_{\lambda,\nu}(q,t)P_{\nu}(q,t)
			\end{eqnarray} we get the inverse statement for $K^{(1)}$. The proof for $(K^{(2)})^{\pm 1}$ is obtained from by replacing $m_{\mu}$s in the above argument by $s_{\mu}$s everywhere.   
		\end{proof}


		\section{Multiplicity-one pairs}\label{sec:whenKlammu=1}
		In this section we finish the proof that \ref{hagdualconj} holds for all pairs $(\lambda,\mu)$ such that either $K_{\lambda,\mu}=1$ or $K_{\mu',\lambda'}=1$. We recall the reduction principle given in the introduction. 
		
		Let $\lambda,\mu \in Par$ such that $\lambda \geq \mu$. We can always decompose $\lambda = (\lambda^1,\ldots,\lambda^r)$ and $\mu = (\mu^1,\ldots,\mu^r)$ where $\lambda^1,\ldots \lambda^r,\mu^1,\ldots,\mu^r$ are partitions such that for $i\in \{1,\ldots r\}$, $\lambda^i \geq \mu^i$ and for $i \in \{1,\ldots r-1\}$, $l(\lambda^i)=l(\mu^i)$. Then for $i=1,2$, by \autoref{rowmult}, \begin{eqnarray}
			(K^{(i)})^{\pm 1}_{\lambda,\mu}(q,t) = \prod_{j=1}^{r}(K^{(i)})^{\pm 1}_{\lambda^j,\mu^j}(q,t)
		\end{eqnarray}
		
		For $j\in \{1,\ldots, r-1\}$ each pair $(\lambda^j,\mu^j)$ contains a common rectangle of row length atleast $\lambda^{j+1}_1$. By \autoref{cor:rectangleinsert} proving the dual Haglund conjecture is equivalent to proving it for each pair $(\lambda^j,\mu^j)$ with these rectangles removed. We can repeat this process to each pair till no more decomposition is possible. 	
		
		Therefore, to prove the dual Haglund's conjecture we only need to prove \ref{hagdualconj} for pairs of partitions $(\lambda,\mu)$ which are \textsl{irreducible} i.e, satisfying $\lambda \geq \mu$ and  $\lambda_1+\ldots+\lambda_i > \mu_1+\ldots+\mu_i$ for $1\leq i \leq \ell(\lambda)$.
		
		\subsubsection{When $K_{\lambda,\mu}=1$}

		Recall the characterization of irreducible pairs $(\lambda,\mu)$ with $K_{\lambda,\mu}=1$ given in \autoref{BZtheorem}. Case 2 of \autoref{BZtheorem} is already done by \autoref{hagforroworcol}. Now consider Case 1 of \autoref{BZtheorem}. 
		Let $\lambda = (m^n)\geq \mu$ and $l(\mu) = n+1$.  Since both $\lambda,\mu \subset (m^{n+1})$, we have by \autoref{complement} that \begin{eqnarray}
			(K^{(i)})^{\pm 1}_{\lambda,\mu}(q,t) = (K^{(i)})^{\pm 1}_{(m),\mu^c}(q,t) \quad \text{for } i = 1,2
		\end{eqnarray} Here $\mu^c = (m-\mu_{n+1},...m-\mu_1)$ is the complement of $\mu$ in the rectangle $(m^{n+1})$ and $(m)$ is the complement of $\lambda$.
		
		In particular, $K^{(2)}_{\lambda,\mu} = K^{(2)}_{(m),\mu^c}$ implies \begin{eqnarray}
			k_{\lambda,\mu}(q,t) = k_{(m),\mu^c}(q,t) \dfrac{c'_{\mu}(q,t)}{c'_{\mu^c}(q,t)}
		\end{eqnarray}

		Note that $c'_{\mu}$ is a product of $|\mu|=mn$ many terms, $c'_{\mu^c}$ is a product of $m(n+1)- |\mu|=m$ many terms. If all terms in the denominator cancels off with terms in $c'_{\mu}$ then $\dfrac{c'_{\mu}(q,t)}{c'_{\mu^c}(q,t)}$ will be a product of $mn-m$ terms each of the form $(1-q^\alpha t^\beta)$ for $\alpha,\beta \in \mathbb{Z}_{\geq 0}$, substituting $q=t^k$ and dividing by $(1-t)$ we get $\dfrac{1-t^{k\alpha+\beta}}{1-t} \in \mathbb{Z}_{\geq 0}[t]$. Since we know $\dfrac{k_{(m),\mu^c}(t^k,t) }{(1-t)^m} \in \mathbb{Z}_{\geq 0}$ by \autoref{hagforroworcol}, this would imply that $\dfrac{k_{\lambda,\mu}(t^k,t)}{(1-t)^{mn}} \in \mathbb{Z}_{\geq 0}[t]$. To prove that $\dfrac{c'_{\mu}(q,t)}{c'_{\mu^c}(q,t)}$ is a product of $mn-m$ terms, each of the form $(1-q^\alpha t^\beta)$ for $\alpha,\beta \in \mathbb{Z}_{\geq 0}$ we consider $f_\mu(q,t) = \sum_{x \in \mu}q^{a(x)}t^{l(x)}$ and show that $f_{\mu}(q,t)-f_{\mu^c}(q,t)\in \mathbb{Z}_{\geq 0}[q,t]$. Since each monomial $q^{\alpha}t^{\beta}$ (counted with multiplicity) that occurs in $f_{\mu}(q,t)-f_{\mu^c}(q,t)$, corresponds to a term $(1-q^{\alpha+1}t^{\beta})$ in $\dfrac{c'_{\mu}}{c'_{\mu^c}}$, this is clearly sufficient.

		\begin{lemma}
			Let $\mu \in Par$. Then \begin{eqnarray}
				f_{\mu}(q,t) = \sum_{j=0}^{l(\mu)-1} t^j \bigg( \sum_{i=1}^{l(\mu)-j} q^{\mu_i-\mu_{i+j}}[\mu_{i+j}-\mu_{i+j+1}]_q \bigg)  
			\end{eqnarray}
		\end{lemma}
		\begin{proof}
			
			For $i \in \{1,\ldots l(\mu)\}$, in the $i$th row of $\mu$ the rightmost $\mu_{i}-\mu_{i+1}$ many boxes have leg length $0$ and arm lengths $0,\ldots,\mu_{i}-\mu_{i+1}-1$. The next $\mu_{i+1}-\mu_{i+2}$ many boxes from the right have leg length $1$ and arm lengths $\mu_i-\mu_{i+1},\ldots,\mu_i-\mu_{i+2}-1$ and so on. Hence the contribution of the $i$th row to $f_{\mu}(q,t)$ is \begin{eqnarray}
				\sum_{j=0}^{l(\mu)-i} t^j q^{\mu_{i}-\mu_{i+j}}[\mu_{i+j}-\mu_{i+j+1}]_q
			\end{eqnarray}
			
			So \begin{eqnarray}
				f_{\mu}(q,t) & = & \sum_{i=1}^{l(\mu)}
				\sum_{j=0}^{l(\mu)-i} t^j q^{\mu_{i}-\mu_{i+j}}[\mu_{i+j}-\mu_{i+j+1}]_q \\ & = & \sum_{j=0}^{l(\mu)-1} t^j \bigg( \sum_{i=1}^{l(\mu)-j} q^{\mu_i-\mu_{i+j}}[\mu_{i+j}-\mu_{i+j+1}]_q \bigg) 
			\end{eqnarray}
			
		\end{proof}
		
		\begin{corollary}\label{fmu-fmuc}
			Let $\mu \subset (m^{n+1})$, with $\mu_1 < m$ and $l(\mu)=n+1$. Let $\mu^c$ denote the complement of $\mu$ in $(m^{n+1})$. Then \begin{eqnarray}
				f_{\mu}(q,t)-f_{\mu^c}(q,t) = \dfrac{1}{1-q} \sum_{j=0}^{n} t^j(q^{\mu^c_{n+1-j}}-q^{\mu_{n+1-j}})
			\end{eqnarray}  
		\end{corollary}
		
		\begin{proof}
			We have \begin{eqnarray}
				f_{\mu}(q,t) = \sum_{j=0}^{n}t^j\bigg( \sum_{i=1}^{n-j}q^{\mu_i-\mu_{i+j}}[\mu_{i+j}-\mu_{i+j+1}]_q+ q^{\mu_{n+1-j}-\mu_{n+1}}[\mu_{n+1}]_q \bigg)
			\end{eqnarray}
			Since $(\mu^c)_i = m - \mu_{n+2-i}$, we have $l(\mu^c)=n+1$ since $\mu_1 <m$. Thus \begin{eqnarray}
				f_{\mu^c}(q,t) & = & \sum_{j=0}^{n+1}t^j \bigg( \sum_{i=1}^{n-j} q^{\mu_{n+2-i-j}-\mu_{n+2-i}}[\mu_{n+1-i-j}-\mu_{n+2-i-j}]_q \nonumber \\ & & \quad \quad \quad \quad + q^{\mu_1-\mu_{n+2-n-1+j}}[m-\mu_1]_q \bigg) \nonumber \\ & = & \sum_{j=0}^{n+1}t^j \bigg(\sum_{i'=1}^{n-j} q^{\mu_{i'+1}-\mu_{i'+j+1}}[\mu_{i'}-\mu_{i'+1}]_q + q^{\mu_1-\mu_{j+1}}[m-\mu_1]_q \bigg) \nonumber 
			\end{eqnarray}
			
			Therefore the coefficient of $t^j$ in $f_{\mu}(q,t)-f_{\mu^c}(q,t)$ is \begin{eqnarray}
				& & \dfrac{1}{1-q} \bigg( \sum_{i=1}^{n-j} (q^{\mu_i-\mu_{i+j}} - q^{\mu_i-\mu_{i+j+1}}) - \sum_{i'=1}^{n-j}(q^{\mu_{i'+1}-\mu_{i'+j+1}}-q^{\mu_{i'}-\mu_{i'+j+1}}) \nonumber \\ & & + q^{\mu_{n+1-j}-\mu_{n+1}} - q^{\mu_{n+1-j}} - q^{\mu_1-\mu_{j+1}} + q^{m-\mu_{j+1}} \bigg) \nonumber	\\ & = & \dfrac{1}{1-q} \bigg( q^{\mu_1-\mu_{1+j}} - q^{\mu_{n+1-j}-\mu_{n+1}} + q^{\mu_{n+1-j}-\mu_{n+1}} - q^{\mu_{n+1-j}} - q^{\mu_1-\mu_{j+1}} + q^{m-\mu_{j+1}} \bigg) \nonumber \\ & = & \dfrac{1}{1-q} (q^{m-\mu_{j+1}}-q^{\mu_{n+1-j}}) = \dfrac{1}{1-q} (q^{\mu^c_{n+1-j}}-q^{\mu_{n+1-j}})  \nonumber
			\end{eqnarray}
		\end{proof}	
		
		\begin{lemma}\label{m-muineq}
			Suppose $\mu \vdash mn$, $l(\mu)=n+1$ and $\mu \subset (m^{n+1})$. Then for all $i \in \{ 1,...,n+1 \}$, $\mu^c_i \leq \mu_{i}$   
		\end{lemma}
		\begin{proof}
			Suppose $\mu^c_i = m-\mu_{n+2-i}>\mu_i  $ for some $i\in \{1,...,n+1\}$.
			
			If $i \neq n+2-i$ then \begin{eqnarray}
				(n-1)\mu_1 \geq \sum_{j \notin \{i,n+2-i\}} \mu_j  = |\mu|-(\mu_i+\mu_{n+2-i}) >(n-1)m
			\end{eqnarray}
			which implies $\mu_1>m$, a contradiction.
			
			If $i = n+2-i$ then $m-\mu_{n+2-i}>\mu_i $ implies $2\mu_i < m$. Then \begin{eqnarray}
				2(n-1)\mu_1 \geq 2 \sum_{j \neq i  }\mu_j = 2|\mu|-2\mu_i > m(2n-1) 
			\end{eqnarray} 
			since $m \geq \mu_1$ we get $2(n-1)>2n-1$, a contradiction.
		\end{proof}
		
		With the assumptions of \autoref{fmu-fmuc}, by \autoref{m-muineq} \begin{eqnarray}
			f_\mu(q,t)-f_{\mu^c}(q,t) \in \mathbb{Z}_{\geq 0}[q,t]
		\end{eqnarray}
		
		\begin{corollary}
			\emph{\ref{hagdualconj}} is true whenever $K_{\lambda,\mu} = 1$
		\end{corollary}
		
		\begin{proof}
			We showed that when $\lambda,\mu$ is an irreducible pair with $K_{\lambda,\mu}=1$, then \ref{hagdualconj} is true. \autoref{cor:rectangleinsert} shows that \ref{hagdualconj} remains true in the patching together of irreducible pairs.
		\end{proof}

		\subsubsection{When $K_{\mu',\lambda'}=1$}
		
		\begin{proposition}
			\emph{\ref{hagdualconj}} is true whenever $K_{\mu',\lambda'}=1$
		\end{proposition}
		\begin{proof}
			
			Suppose $\lambda,\mu \in Par$ such that $K_{\mu',\lambda'}=1$. To compute $k_{\lambda,\mu}(q,t)$ we use $(K^{(2)})^{-1}_{\mu',\lambda'}(t,q)$ in this case (\autoref{k=K2c'}). As before by use of \autoref{rowmult} and \autoref{colmult} we reduce the calculation of $(K^{(2)})^{-1}_{\mu',\lambda'}(t,q)$ to the cases when $\mu',\lambda'$ is an irreducible pair. The case $\mu'=(m)$ for some $m$ is done by \autoref{hagforroworcol}. Let $\mu' = (m^n), l(\lambda') = n+1$. We have by \autoref{complement} \begin{eqnarray}
				(K^{(2)})^{-1}_{\mu',\lambda'}(t,q) = (K^{(2)})^{-1}_{(m),\lambda'^{c}}(t,q)
			\end{eqnarray}    
			
			Hence by \autoref{k=K2c'}, \begin{eqnarray}
				k_{\lambda,\mu}(q,t) = k_{\lambda^c,(1^m)}(q,t) \dfrac{c_{(m^n)}(q,t)}{c_{(m)}(q,t)}	\end{eqnarray}
			
			In $\dfrac{c_{(m^n)}(q,t)}{c_{(m)}(q,t)}$, the contribution from the last row of $(m^n)$ cancels with the contribution from $(m)$, so it is a product of $mn-1$ terms each of the form $1-q^{\alpha}t^{\beta}$ for some $\alpha,\beta \in \mathbb{Z}_{\geq 0}$ and since by \autoref{hagforroworcol}, $\dfrac{k_{\lambda^c,1^m}(t^k,t)}{(1-t)^m} \in \mathbb{Z}_{\geq 0}[t]$ this implies \ref{hagdualconj} for $\lambda,\mu$, and hence for all pairs $\lambda,\mu$ with $K_{\mu',\lambda'}=1$ when $\mu',\lambda'$ is irreducible. Applying \autoref{cor:rectangleinsert} we conclude the statement for general $\lambda,\mu$ with $K_{\mu',\lambda'}=1$

		\end{proof}

		\end{document}